\documentclass[11pt, leqno]{amsart}
\usepackage{geometry}                
\usepackage{graphicx}
\usepackage{amssymb}
\usepackage{epstopdf}
\usepackage{amsmath}
\usepackage{amsthm}
\usepackage{amsfonts}
\usepackage{amsxtra}     
\usepackage{epsfig}
\usepackage{verbatim}

\DeclareMathOperator{\Vol}{Vol}
\DeclareMathOperator{\Out}{Out}
\DeclareMathOperator{\Isom}{Isom}

\providecommand{\abs}[1]{\lvert#1\rvert}
\providecommand{\norm}[1]{\lVert#1\rVert}

\newtheorem{theorem}{Theorem}
\newtheorem{prop}[theorem]{Proposition}
\newtheorem{corollary}[theorem]{Corollary}
\newtheorem{lemma}[theorem]{Lemma}

\newtheorem{assumption}[theorem]{Assumption}
\newtheorem{claim}[theorem]{Claim}

\newtheorem{definition}[theorem]{Definition}

\newtheorem{remark}[theorem]{Remark}

\title{Lower bounds for the volume of hyperbolic $n$-orbifolds}
\author{Ilesanmi Adeboye}
\date{}                                           

\begin{document}
\begin{abstract}
In this paper an explicit formula for a lower bound on the volume of a hyperbolic orbifold, dependent on dimension and the maximal order of torsion in the orbifolds' fundamental group, is constructed.
\end{abstract}
\bibliographystyle{plain}
\maketitle
\thispagestyle{empty}
\section{Introduction}A complete orientable \textit{hyperbolic n-orbifold} is an orbit space $\mathbb{H}^n/\Gamma$, where $\Gamma$ is a discrete group of orientation-preserving isometries of $\mathbb{H}^n$. An orientable hyperbolic $n$-manifold is the quotient of $\mathbb{H}^n$ by a discrete \textit{torsion-free} subgroup of $\Isom_{+} (\mathbb{H}^n)$. Explicit lower bounds for the volume of a  hyperbolic $3$-manifold, as well as for the volume of a hyperbolic $3$-orbifold, were given by Meyerhoff \cite{Mey1}. Later, explicit bounds for manifolds in all dimensions were constructed by Martin \cite{Mar1} and Friedland and Hersonsky \cite{FriHer}. Wang's finiteness theorem \cite{Wa} asserts that, for $n$ greater than three, the set of volumes of a hyperbolic $n$-orbifolds is discrete in the real numbers. Hence, lower bounds of the volume of hyperbolic $n$-orbifolds exist in all dimensions. In this paper we prove the following result.

\begin{theorem}[Main Theorem]
\label{MainTheorem}
Let $\Gamma$ be a discrete group of orientation-preserving isometries of $\mathbb{H}^n$. Assume that $\Gamma$ has no torsion element of order greater than $k$. Then
\[\Vol(\mathbb{H}^n/\Gamma) \geq \mathcal{A}(n,k)\]
where $\mathcal{A}(n,k)$ is an explicit constant depending only on $n$ and $k$. 
\end{theorem}

More precisely, \[\mathcal{A}(n,k)=\sup_{r>0}\left(1+\left(\frac{e(n+1)(1+\cosh r)}{\sinh r}\right)^{2}\cosh 6r\sin^{-2}\left(\frac{\pi}{k}\right)\right)^{-(n+1)^2}       
\int_{0}^{r}n\frac{\pi^{\frac{n}{2}}}{(n/2)!}\sinh^{n-1}(u)du. \]

As a corollary, we obtain the following analogue of Hurwitz's formula for groups acting on surfaces.

\begin{corollary} \label{corollary 1} If $M$ is an orientable hyperbolic $n$-manifold and $G$ is a group of orientation preserving isometries of $M$ containing no torsion elements of order greater than $k$, then

\[|G| \leq \frac{\Vol(M)}{\mathcal{A}(n,k)}.\qed
\] \end{corollary}
I would like to thank Professors Francis Bonahon and Mario Bonk for many helpful discussions. I am very grateful to Professor Dick Canary for reading several versions of this paper, as well as for his encouragement, support and mentoring over the years.

\section{Preliminaries}
\label{123}
We denote \textit{hyperbolic $n$-space} by $\mathbb{H}^n$ and define it as \[\mathbb{H}^n=\left\{(x_1,\dotsc,x_{n+1}) \in \mathbb{R}^{n+1}: -x_1^2+x_2^2+\dotsb +x_{n+1}^2=-1,x_1>0\right\}\] together with the Riemannian metric induced on $\mathbb{H}^n$ by the quadratic form $ds^2=-dx_1^2+dx_2^2\dotsb+dx_{n+1}^2.$
The Riemannian metric gives rise to a distance function. Given two vectors $x,y \in \mathbb{H}^n$ the \textit{hyperbolic distance} between $x$ and $y$ is denoted by $d_{\mathbb{H}}(x,y)$ and defined by the equation 
\[\cosh d_{\mathbb{H}}(x,y) = x_1y_1-\dotsb-x_{n+1}y_{n+1}.\]Let $e_i$ denote the standard basis element. We will make particular use of $e_1=(1,0,\dotsc,0)$ which is an element of $\mathbb{H}^n.$ The group of isometries of hyperbolic space will be identified with the Lie group $O^{+}(1,n)$ \cite{Bea}. The subgroup $SO^{+}(1,n)$, consisting of all elements of $O^{+}(1,n)$ with determinant $1$, corresponds to orientation-preserving isometries of $\mathbb{H}^n$. The symbol $I_n$ will denote the $n\times n$ identity matrix. The \textit{torsion} elements of a discrete group of isometries of hyperbolic space, that is isometries of finite order, are called \textit{elliptic}. We will use the two terms interchangeably. 

For an element $A$ of $O^{+}(1,n)$ we define its operator norm to be \[\norm{A}=\max\left\{\abs{Av}:v\in \mathbb{R}^{n+1} \text{ and } \abs{v}=1\right\}.\] There is an important alternative definition for the operator norm. Let $A$ be any $n\times n$ matrix. The \textit{spectrum}, denoted by $\sigma(A)$, is the set of all eigenvalues of $A$. The \textit{spectral radius}, denoted by $r_{\sigma}(A)$, is defined by the equation \[r_{\sigma}(A)=\max_{\lambda\in\sigma(A)}\abs{\lambda}.\]It is proved in section 7.3 of \cite{Atk} that \[\norm{A}=\sqrt{r_{\sigma}(A^{t}A)}.\]

The \textit{conformal ball model} of hyperbolic $n$-space consists of $\mathbb{B}^n$, the open unit ball in $\mathbb{R}^n$, together with the metric \[ds^{2}_{\mathbb{B}}=\frac{4(dx^2_1+\cdots+dx^2_n)}{(1-\abs{x}^2)^2}.\]
\bigskip

\noindent \textbf{Outline.} The proof of the Main Theorem will come in three steps. In section~\ref{OrigThm1} we prove that an upper bound on the order of an elliptic isometry $A$ of $\mathbb{H}^n$ gives a lower bound on the operator norm of $A-I_{n+1}$. In section~\ref{OrigThm2} we show that, up to conjugation, an upper bound on the maximal order of torsion of a discrete group of isometries of $\mathbb{H}^n$ leads to a uniform lower bound on $\|A-I_{n+1}\|$ for all $A\neq I_{n+1}$. Finally, in section~\ref{OrigThm3}, we establish an upper bound on the number of elements of a discrete group of isometries of $\mathbb{H}^n$ that fail to move a ball of radius $r$ of itself. This allows us to bound the volume of the image of such a ball in the orbit space.

\section{Norm Bound for Low-Order Torsion Elements}
\label{OrigThm1}

In this section we prove the following proposition:

\begin{prop}
\label{origthm1}
Let $A\in O^{+}(1,n)$ be an elliptic element of order at most $k$. Then \[\|A-I_{n+1}\| \geq c_{k}\] where $c_k :=2\sin^2\left(\frac{\pi}{k}\right)e^{-2}.$
\end{prop}

The first step is to prove a version of Proposition~\ref{origthm1} for the elements of the subgroup of $O^{+}(1,n)$ that fix $e_{1}$. Next, we will consider the remaining elliptic  elements, which are all conjugate to elements which fix $e_{1}$. Two different bounds on $\|A-I_{n+1}\|$ that depend on the distance between the fixed point set of $A$ and $e_{1}$ will be developed. Finally, all results will be combined to prove the proposition.


Define $E(n)$  to be the subgroup of $O^{+}(1,n)$ that stabilizes the vector $e_{1}$. We identify $E(n)$ with the \textit{orthogonal group} $O(n)$ by noting that for each $A \in E(n)$, there exists $A^{*} \in O(n)$  such that
\[ A = \begin{pmatrix}
1&  \\
 &A^{*}
\end{pmatrix}\]

Using this identification we may carry over properties of $O(n)$ to $E(n)$. In particular
\begin{equation}
\label{in-tr}
A \in E(n) \Longrightarrow A^{-1} = A^{t}.
\end{equation}
A basic result of linear algebra, the proof of which, for instance, can be found in section 6.4 of \cite{LinGeo}, gives us the following lemma.

\begin{lemma} 
\label{basic}
Given $A \in E(n)$ there exists $B \in E(n)$ such that    
\begin{equation}
\label{stamx}
BAB^{-1} = \begin{pmatrix}   
1 &  &  &  &  &  &  &  \\ &    A_{1} &  &  &  &  &  &  \\ &  &    . &  &  &  &  &  \\ &  &  &    . &  &  &  &  \\ &  &  &  &    . &  &  &  \\ &  &  &  &  &    A_{l} &  &  \\ &  &  &  &  &  &    -I_{s} &  \\ &  &  &  &  &  &  &    I_{t}\end{pmatrix}
\end{equation}
where l, s, t are non-negative integers and
\begin{equation*}
A_i = \begin{pmatrix}\cos \theta_{i} & -\sin \theta_{i} \\\sin \theta_{i} & \cos \theta_{i}   \end{pmatrix}
\end{equation*} 
 where $0 < \theta_{i}  < \pi$.$\qed$ \end{lemma} 
 
 \begin{remark}
 \label{thetabound}
If in Lemma~\ref{basic} the order of $A$ is at most $k$, then clearly
\[\frac{2\pi}{k} \leq \theta_{i} < \pi.\]
\end{remark}

Since the set of eigenvalues of a matrix is conjugacy invariant, so is the operator norm.
\begin{lemma}
\label{similar}
 Let $A$ be a real $(n+1) \times (n+1)$ matrix and let $B \in E(n)$, then \[\norm{BAB^{-1}} = \norm{A}.\qed\] \end{lemma}

The following proposition gives our result for elements of $E(n)$.

\begin{prop} 
\label{E(n)result}
If $A$ is a non-identity element of $E(n)$ of order at most $k$, then 

\[\|A-I_{n+1}\| \geq 2\sin\left(\frac{\pi}{k}\right).\]

\end{prop}

\begin{proof} Write $A = BA^{\prime}B^{-1}$, where $A^{\prime}$ has the form of the right-hand side of equation~\ref{stamx} and $B$ is the appropriate element of $E(n)$. Then by Lemma~\ref{similar}
\begin{equation*}
\begin{split}
\norm{A-I_{n+1}}&= \norm{BA^{\prime}B^{-1} - I_{n+1}}\\
&= \norm{B(A^{\prime} - I_{n+1})B^{-1}}\\ 
&= \norm{A^{\prime} - I_{n+1}}\,\,\,\text{(by Lemma~\ref{similar}})\\ 
&\geq \abs{(A^{\prime} - I_{n+1})e_{2}}\\
&= \abs{A^{\prime}e_{2} - e_{2}}.
\end{split}
\end{equation*}

Assume $l \neq 0$, then
\[\abs{A^{\prime}e_{2} - e_{2}}= \sqrt{(\cos\theta_{1} - 1)^2 + \sin^{2}\theta_{1}} = 2\sin\left(\frac{\theta_{1}}{2}\right)\geq 2\sin\left(\frac{\pi}{k}\right)\]
by Remark~\ref{thetabound}. If $l = 0$, then $ k = 2$ and $A^{\prime}$ has the form
\[\begin{pmatrix} 1 &  &  \\ &    -I_{s} &  \\ &  &    I_{t}\end{pmatrix},\] therefore
\[\abs{A^{\prime}e_{2} - e_{2}}= \abs{- e_{2} - e_{2}}= 2 = 2\sin\left(\frac{\pi}{2}\right).\] \end{proof}


Now consider the more general case where $A$ is an elliptic element of $O^{+}(1,n)$ which does not fix $e_{1}$. Our first approach will give us a bound in the case where the fixed point set of $A$ is  ``close" to $e_{1}$.  We will show that $A$ is conjugate, by an isometry whose norm we can explicitly calculate, to an elliptic element of the same order which fixes $e_{1}$. Proposition~\ref{E(n)result} can then be used to obtain a bound on $\|A -I_{n+1}\|$. 

\begin{prop} 
\label{conE(n)result1}
Let $A\in O^{+}(1,n)$ be an elliptic element of order at most $k$, which does not fix $e_{1}$. Let $\delta$ be the hyperbolic distance from $e_{1}$ to the fixed point set  of $A$. Then
\[\norm{A -I_{n+1}} \geq2\sin\left(\frac{\pi}{k}\right)e^{-2\delta}.\]
\end{prop}

\begin{proof} Let $b$ be the closest point in the fixed point set of $A$ to $e_{1}$. Let $\hat{b} = (\cosh \delta, \sinh \delta, 0,\dotsc ,0)$. Then $d_{\mathbb{H}}(e_1,b)=d_{\mathbb{H}}(e_1,\hat{b})=\delta$. Therefore, there exists $\hat{A}\in E(n)$ such that $\hat{A}\hat{b}=b.$

Let \[T = \begin{pmatrix} \cosh \delta& \sinh \delta &  &  &  &  \\\sinh \delta  & \cosh \delta &  &  &  &  \\ &  & 1 &  &  &  \\ &  &  &    . &  &  \\ &  &  &  &    . &  \\ &  &  &  &  & 1\end{pmatrix},\] define $B:=\hat{A}T$ and let $\tilde{A}=B^{-1}AB$. Since $\tilde{A}$ fixes $e_1$, $\tilde{A}\in E(n)$ and 


\begin{equation*}
\begin{split}
\norm{\tilde{A} - I_{n+1}}&= \norm{B^{-1}AB -I_{n+1}}\\
&= \norm{B^{-1}(A - I_{n+1})B}\\ 
&\leq \norm{B^{-1}}\norm{A - I_{n+1}}\norm{B}\\ 
&= \norm{B}^{2}\norm{A -I_{n+1}}\,\,\left(\text{ since }\norm{B}=\norm{B^{-1}}\text{ for all }B \in O^{+}(1,n)\right).
\end{split}
\end{equation*}


Furthermore
\begin{equation*}
\begin{split}
B^{t}B &= (\hat{A}T)^{t}(\hat{A}T)\\ 
&= T^{t}\hat{A}^{t}\hat{A}T\\ 
&= T^{t}\hat{A}^{t}\hat{A}T\text{ (by equation~\ref{in-tr})}\\ 
&= T^{t}T\\ 
&= \begin{pmatrix} \cosh 2\delta & \sinh 2\delta &  &  &  &  \\\sinh 2\delta  & \cosh 2\delta &  &  &  &  \\ &  & 1 &  &  &  \\ &  &  &    . &  &  \\ &  &  &  &    . &  \\ &  &  &  &  & 1\end{pmatrix}.
\end{split}
\end{equation*}

The eigenvalues of $B^{t}B$ are $1, e^{-2\delta}$ and $e^{2\delta}$, thus $\norm{B} = e^{\delta}.$ Therefore
\begin{equation*}
\begin{split}
\norm{A -I_{n+1}} &\geq \norm{\tilde{A} - I_{n+1}}\norm{B}^{-2}\\ 
&= \norm{\tilde{A} - I_{n+1}}e^{-2\delta}\\ 
&\geq 2\sin(\frac{\pi}{k})e^{-2\delta}\text{    (by Proposition~\ref{E(n)result})}.
\end{split}
\end{equation*} \end{proof}


In Proposition~\ref{conE(n)result1} our bound goes to zero as $\delta \rightarrow \infty$. This is our estimate for ``small'' values of $\delta$. The following proposition uses a different method to address the case where $\delta$ is ``large". If the fixed point set of an elliptic element $A$ is far from $e_1$ then $A$ must move $e_1$ a significant amount. This ensures, by the definition of the operator norm, that $\norm{A-I_{n+1}}$ can be bounded away from zero.

\begin{prop} 
\label{conE(n)result2}
Let $A \in O^{+}(1,n)$ be an elliptic element of order at most $k$, which does not fix $e_{1}$. Let $\delta$ be the hyperbolic distance from $e_{1}$ to the fixed point set  of $A$. Then
\[\norm{A - I_{n+1}} \geq 2\sinh^{2}\delta\sin^{2}\left(\frac{\pi}{k}\right).\]
\end{prop}
The proof of the following lemma is straight-forward.

\begin{lemma} 
\label{lemmaforCon2}
Let $A$ be an elliptic isometry of $\mathbb{H}^{n}$ of order at most $k$. Let $b$ be a fixed point of $A$ and let $v$ be an element of $\mathbb{H}^{n}$ such that the geodesic $g_{1}$ containing $v$ and $b$ is perpendicular to the fixed point set of $A$. Suppose $g_{2}$ is the geodesic containing $b$ and $A(v)$, then the angle of intersection $\theta$ between $g_{1}$ and $g_{2}$ is greater than or equal to $\frac{2\pi}{k}$.$\qed$
\end{lemma}

\emph{Proof of Proposition~\ref{conE(n)result2}:}

Let $A = (a_{ij})$. By the definition of the operator norm we have the following
\begin{eqnarray*}
\|A -I_{n+1}\| &\geq& |(A - I_{n+1})e_{1}|\\ & = & |Ae_{1} - e_{1}| \\
                                                                    & =  & \left|(a_{11}-1, a_{21}, \dots ,a_{(n+1)1} )\right| \\
                                                                    & \geq & |a_{11} - 1|.\end{eqnarray*} On the other hand $\cosh d_{\mathbb{H}}(e_{1}, Ae_{1}) = 1\cdot a_{11} = a_{11}$. Therefore \[\|A -I_{n+1}\| \geq |\cosh d_{\mathbb{H}}(e_{1},Ae_{1}) - 1|.\]

Let $b$ be the point in the fixed point set of $A$ closest to $e_{1}$. We can apply hyperbolic cosine rule to the triangle with vertices $e_{1}$, $Ae_{1}$ and $b$. If $\theta$ is the angle of the triangle at the point $b$, then \[\cosh d_{\mathbb{H}}(e_{1},Ae_{1}) = \cosh d_{\mathbb{H}}(b,e_{1})\cosh d_{\mathbb{H}}(b,Ae_{1}) - \sinh d_{\mathbb{H}}(b,e_{1})\sinh d_{\mathbb{H}}(b,Ae_{1})\cos\theta.\]
Therefore,\begin{eqnarray*}
\cosh d_{\mathbb{H}}(e_{1},Ae_{1}) & = & \cosh^{2}\delta - \sinh^{2}\delta\cos\theta\\ 
&=& \cosh^{2}\delta - \sinh^{2}\delta(1 - 2\sin^2(\theta/2)) \\
                                    & = & \cosh^{2}\delta - \sinh^{2}\delta + 2\sinh^{2}\delta\sin^{2}(\theta/2) \\
                                    & = & 1 + 2\sinh^{2}\delta\sin^{2}(\theta/2).
                                      \end{eqnarray*}
Hence,
\begin{eqnarray*}
\|A -I_{n+1}\| &\geq& |1 + 2\sinh^{2}\delta\sin^{2}(\theta/2) - 1|\\
 & = & 2\sinh^{2}\delta\sin^{2}(\theta/2) \\
  & \geq & 2\sinh^{2}\delta\sin^{2}(\pi/k)\mbox{  (by Lemma~\ref{lemmaforCon2})}.\,\,\qed
\end{eqnarray*}

The following lemma follows immediately from Propositions~\ref{conE(n)result1} and~\ref{conE(n)result2}.

\begin{lemma}
Let $A\in O^{+}(1,n)$ be an elliptic element of order at most $k$. Let $\delta$ be the hyperbolic distance from $e_{1}$ to the fixed point set  of $A$. Then \[\|A - I_{n+1}\| \geq \max \left\{2\sinh^{2}\delta\sin^{2}\left(\frac{\pi}{k}\right), 2\sin\left(\frac{\pi}{k}\right)e^{-2\delta}\right\}.\qed\]\end{lemma}

Hence, we have
 \begin{lemma}
 \label{setup1}
Let $A\in O^{+}(1,n)$ be an elliptic element of order at most $k$. Then 
\[\|A - I_{n+1}\| \geq \inf_{\delta>0} \max \left\{2\sinh^{2}\delta\sin^{2}\left(\frac{\pi}{k}\right), 2\sin\left(\frac{\pi}{k}\right)e^{-2\delta}\right\}.\qed\]\end{lemma}

\begin{proof}[Proof of Proposition~\ref{origthm1}] We must show that the bound of Lemma~\ref{setup1} agrees with the uniform bound $c_{k}$. We divide the proof into two cases. First, assume $\delta\geq 1$. Since $\sinh^{2}(\delta)$ is an increasing function, we have that 

\[ \begin{array}{lcl}
2\sinh^{2}(\delta)\sin^{2}(\frac{\pi}{k})&\geq& 2\sinh^{2}(1)\sin^{2}(\frac{\pi}{k}) \\
 &\geq& 2\sin^{2}(\frac{\pi}{k})e^{-2}. \end{array} 
\]

Now assume $\delta\leq 1$. Note that $e^{-2\delta}$ is a decreasing function. Also note that since $A$ is non-trivial, $k > 1$. Therefore $\sin^{2}(\frac{\pi}{k})\leq \sin(\frac{\pi}{k})$. Hence,

\[ \begin{array}{lcl}
2\sin(\frac{\pi}{k})e^{-2\delta}&\geq& 2\sin^{2}(\frac{\pi}{k})e^{-2\delta} \\
 &\geq& 2\sin^{2}(\frac{\pi}{k})e^{-2}. \end{array} 
\]

Therefore, by Lemma~\ref{setup1},

\[ \begin{array}{lcl}
\|A - I_{n+1}\| &\geq& \inf_{\delta>0} \max 
\left\{2\sinh^{2}\delta\sin^{2}(\frac{\pi}{k}), 
2\sin(\frac{\pi}{k})e^{-2\delta}\right\} \\
 &\geq& 2\sin^{2}(\frac{\pi}{k})e^{-2}. \end{array} 
\]
And we are done.\end{proof}

\section{Norm Bound for Low-Order Torsion Groups}
\label{OrigThm2}
On his way to providing lower bounds on the volume of hyperbolic manifolds Martin proved in \cite{Mar1} the following theorem.

\begin{theorem}[Martin] \label{MarBalls} Let $\Gamma$ be a discrete non-elementary torsion free subgroup of $O^{+}(1,n)$. Then there is an $\alpha \in O^{+}(1,n)$ such that
\[\|A\|\,\|A - I_{n+1}\| \geq \frac{1}{2\sqrt{2}}\]
for all $A \in \alpha\Gamma\alpha^{-1}$, $A \neq I_{n+1}$.$\qed$\end{theorem}

In this section we prove an orbifold version of this result. That is we drop the condition that $\Gamma$ is torsion free and replace it with a bound on the order of torsion. Our proof is similar in outline to Martin's, with Proposition~\ref{origthm1} allowing us to control elliptic elements. 

Recall from the beginning of section~\ref{OrigThm1}, $c_k :=2\sin^2\left(\frac{\pi}{k}\right)e^{-2}.$ The following proposition is our orbifold version of Theorem~\ref{MarBalls}.

\begin{prop}  
\label{origthm2}
Let $\Gamma$ be a discrete non-elementary subgroup of $SO^{+}(1,n)$ which contains no torsion elements of order greater than $k$. Then there is an $\alpha \in O^{+}(1,n)$ such that
$$
\|A-I_{n+1}\| \geq c_{k} 
$$
for all $A \in \alpha \Gamma\alpha^{-1}$, $A\neq I_{n+1}$.\end{prop}

The proof will be by contradiction. In proving Proposition~\ref{origthm2}, we will pass between the ball and hyperboloid models of hyperbolic $n$-space. 


As $\Gamma$ is discrete, it is countable. Therefore $\Gamma$ has a countable number of parabolic and hyperbolic fixed points on the boundary of $\mathbb{B}^{n}$. Furthermore the set of fixed points of each elliptic element on the boundary  of $\mathbb{B}^{n}$ form at most a co-dimension 2 subspace of $\mathbb{S}^{n-1}$. Therefore, we may assume (by conjugation) that no element of the group $\Gamma$ fixes the north or south poles  (resp.~$N, S$) of $\mathbb{S}^{n-1}$.
~For each $t > 0$ let  $\alpha_t$ represent, in the ball or hyperboloid model, the hyperbolic isometry that corresponds to a pure translation by $t$ in the geodesic from $S$ to $N$.

\begin{lemma} 
\label{21}
For each fixed $A  \in \Gamma - \{I_{n+1}\}$, $ \lim_{t \rightarrow \infty}\|\alpha_t A\alpha_t^{-1}\| = \infty$.
\end{lemma}
      
\begin{proof} Let  $\vec{0}$ be the origin in $\mathbb{B}^n$. Note that $\lim_{t \rightarrow \infty}\alpha_t^{-1}(\vec{0}) = S$. Hence,  for $A \in \Gamma-\{I_{n+1}\}, \lim_{t\rightarrow \infty}A\alpha_t^{-1}(\vec{0}) = A(S)$. Since $A(S) \neq S$, there exist a neighborhood $V$ of $S$ in $\bar{\mathbb{B}}^{n}$ such that $A(S) \notin V$. Thus, $\lim_{t\rightarrow \infty}\alpha_tA\alpha_t^{-1}(\vec{0})=N$. Transferring from the ball model to the hyperboloid model, we have $\lim_{t\rightarrow \infty}|\alpha_tA\alpha_t^{-1}(e_{1})| = \infty$. Therefore $\lim_{t\rightarrow \infty}\|\alpha_tA\alpha_t^{-1}\| = \infty$.\end{proof}

For the remainder of this section we will work under the assumption that Proposition~\ref{origthm2} fails. That is 
\begin{assumption}
\label{assumption2}
Let $\Gamma$ be a discrete non-elementary subgroup of $SO^{+}(1,n)$ which contains no torsion elements of order greater than $k$. We assume that for all $\alpha \in O^{+}(1,n)$ there exists $A \in \alpha\Gamma\alpha^{-1}$, $A \neq I_{n+1}$, such that
$$
\|A -I_{n+1}\| < c_{k}.
$$\end{assumption}

Under this assumption, we will construct an infinite sequence $\{A_{i}\}$ of elements of $\Gamma$ and a diverging sequence $\{t(i)\}$ of positive real numbers, so that \[\|\alpha_{t(i)}A_{i}\alpha_{t(i)}^{-1} - I_{n+1}\| < c_{k}\,\,\,{\rm{and}}\,\,\,\|\alpha_{t(i)}A_{i+1}\alpha_{t(i)}^{-1} - I_{n+1}\| < c_{k}\mbox{  for all $i$}.\] We then use Martin and Friedland-Hersonsky's generalization of J{\o}rgensen's inequality to show that $\{\|\alpha_{t(i)} A_i\alpha_{t(i)}^{-1}\|\}$ is unbounded and obtain a contradiction.
\begin{definition}
\label{UA}
For  each $A \in \Gamma -\{I_{n+1}\}$, let $U_A : = \{t > 0 : \|\alpha_{t}A\alpha_{t}^{-1} - I_{n+1}\| < c_{k}\}$.\end{definition}

It is clear from Lemma~\ref{21} and Assumption~\ref{assumption2} that $\mathcal{U} = \{U_{A}\}_{A\in\Gamma - \{I_{n+1}\}}$ forms an open cover of the positive real line by bounded sets.




\begin{lemma}
\label{23}
There exist a sequence $t(i) \rightarrow \infty$ and a sequence $U_{A_i}\in \mathcal{U}$ such that $t(i) \in U_{i} \bigcap U_{i+1}$.\end{lemma}

\begin{proof} For each $A \in \Gamma -\{I_{n+1}\}$, let 
$\mathcal{U}_{[1,2]}:= \{U_{A}\bigcap[1,2]\}$. Then, $\mathcal{U}_{[1,2]}$ is an open cover of $[1,2]$. For $x,y \in [1,2]$, we say $x \sim y$ if there exists $r,s\in \mathbb{Z}^{+}$ and a sequence $\{U^{1}_{A_i}\}^{s}_{i = r} \subset \mathcal{U}_{[1,2]}$ such that $x\in U^{1}_{A_r}$, $y\in U^{1}_{A_s}$ and $U^{1}_{A_i}\bigcap U^{1}_{A_i+1}\neq \emptyset$ for all $i$, $r \leq i \leq s-1$.
             
 Reflexivity, symmetry and transitivity are immediate from the fact that $\mathcal{U}_{[1,2]}$ is an open cover. Therefore $\sim$ is an equivalence relation on $[1,2]$. Now, let $E \subset [1,2]$ be an equivalence class and let $x \in E$. There exists an open set $U^{1}_{A_i} \in \mathcal{U}_{[1,2]}$ such that $x \in U^{1}_{A_i}$. By definition of our equivalence relation, if $y \in U^{1}_{A_i}$ then $y \in E$. Therefore we have that $U^{1}_{A_i} \subset E$. Thus $E$ is an open set.  Hence, $[1,2]$ can be divided into disjoint, open equivalence classes. Since $[1,2]$ is connected, there is only one equivalence class.
 
 Since $1 \sim  2$ there exists an $m_{1} \in \mathbb{Z}^{+}$ and a sequence $\{U^{1}_{A_i}\}^{m_{1}}_{i = 1} \subset \mathcal{U}_{[1,2]}$ such that $1 \in U^{1}_{A_1}$, $2 \in U^{1}_{A_{m_{1}}}$ and $U^{1}_{A_i}\bigcap U^{1}_{A_{i+1}}\neq \emptyset$ for all $i$, $1 \leq i \leq m_{1} -1$.  Define $t(1) := 1$, $t(m_1) := 2$ and select $t(i)$ from $U^{1}_{A_{i-1}}\bigcap U^{1}_{A_i}$ for $2 \leq i \leq m_{1} -1$. 
 
 
 Now consider $\mathcal{U}_{[2,3]} = \{U^{2}_{A}\}$, where $U^{2}_{A} := U_{A}\bigcap[2,3]$.  By repeating the program above, we can define $t(i)$ for $m_{1}+1 \leq i \leq m_{2}$, where $m_{2}$ is an integer larger than $m_{1}$ and $t(m_{2}) = 3$. We can then define the corresponding $U_{A_i}$ for $m_{1}+1 \leq i \leq m_{2}$.  
 
 Continuing in this way, we define the required sequences.\end{proof}
 


The next lemma will be key in what follows.

\begin{lemma} 
\label{keylemma}$\|\alpha_{t(i)}A_{i}\alpha_{t(i)}^{-1} - I_{n+1}\| < c_{k}\,\,\,{\rm{and}}\,\,\,\|\alpha_{t(i)}A_{i+1}\alpha_{t(i)}^{-1} - I_{n+1}\| < 
c_{k}\mbox{  for all $i$}
$\end{lemma}

\begin{proof}As $t(i) \in U_{A_{i}}\bigcap U_{A_{i+1}}$ the lemma follows immediately from Definition~\ref{UA}.\end{proof}


Next, we prove that the set $\{\|\alpha_{t(i)}A_{i}\alpha_{t(i)}^{-1}\|\}$ is unbounded. This will be shown to contradict Lemma~\ref{keylemma} and thus establish Proposition~\ref{origthm2}. The following is a special case of Theorem 2.11 in \cite{FriHer} and Theorem 4.5 in \cite{Mar2}.

\begin{theorem}
\label{Her}
Let $\Gamma \subset O^{+}(1,n)$ be a discrete group. Let $\tau$ be the unique positive solution of the cubic equation $2\tau(1 + \tau)^{2} = 1$. If $A,B \in \Gamma$ such that $\langle A, B \rangle$ is a discrete group and 
$$
\|A -I_{n+1}\| < \tau \,\,\,\,,\,\,\,\,\|B - I_{n+1}\| < \tau
$$
then $\langle A, B \rangle$ is a nilpotent group.\end{theorem}

\begin{remark}
\label{tauc}
We note here that
\[ \begin{array}{lcl}
\tau &>& 0.2971 \\
&>& 2e^{-2}\\
 &\geq& 2\sin^{2}(\frac{\pi}{k})e^{-2}\mbox{ for all $k$}\\
 &=&  c_{k}. \end{array} 
\]\end{remark}

We now prove the following claim.

\begin{claim}
\label{claim2}
The set $\{\|\alpha_{t(i)} A_i\alpha_{t(i)}^{-1}\|\}$ is unbounded.\end{claim}

\begin{proof} Since $\alpha_{t(i)}A_{i}\alpha_{t(i)}^{-1}$ and $\alpha_{t(i)}A_{i+1}\alpha_{t(i)}^{-1}$ are elements of the discrete group $\alpha_{t(i)}\Gamma\alpha_{t(i)}^{-1}$, $\langle \alpha_{t(i)}A_{i}\alpha_{t(i)}^{-1},\alpha_{t(i)}A_{i+1}\alpha_{t(i)}^{-1} \rangle$ is a discrete group. By Lemma~\ref{keylemma}, Theorem~\ref{Her} and Remark~\ref{tauc} $\langle \alpha_{t(i)}A_{i}\alpha_{t(i)}^{-1},\alpha_{t(i)}A_{i+1}\alpha_{t(i)}^{-1} \rangle$ is nilpotent and thus elementary. Therefore $\langle A_{i},A_{i+1}\rangle$ is discrete and elementary. 

By Assumption~\ref{assumption2}, if $A_{i}$ is elliptic it has order at most $k$. This implies $$\|\alpha_{t(i)} A_i\alpha_{t(i)}^{-1} - I_{n+1}\| \geq c_{k}$$ by Proposition~\ref{origthm1}. However that directly contradicts the definition of $U_{A_i}$. We conclude that no element of $\{A_{i}\}$ is elliptic. So for all $i$, $A_{i}$ and $A_{i+1}$ are either both parabolic sharing a common fixed point or both hyperbolic sharing a common axis. Therefore, either each $A_{i}$ is hyperbolic or each $A_{i}$ is parabolic and the $A_ i$ all have a common fixed point set.

Let $\Delta$ be the subgroup of $\Gamma$ generated by all $A_{i}$. The group $\Gamma$, and therefore $\Delta$, acts as a discrete group of M\"obius transformation of $\mathbb{B}^{n}$, which is, itself, a subgroup of the group of M\"obius transformations  of $\bar{\mathbb{R}}^{n}$. Thus $\Delta$ is a \textit{convergence group} \cite{GeM}.   

The set $\{A_{i}\}$ is an infinite sequence in $\Delta$, since each $U_A$ is bounded and $t_i\rightarrow\infty$, so by Theorem~3.7 in \cite{GeM} there exists a subsequence of $\{A_{i}\}$, call it $\{A_{i_{j}}\}$, and points $x_{0}$, $y_{0}$ $\in \bar{\mathbb{R}}^{n}$ such that
$$
\lim_{j\rightarrow \infty}A_{i_{j}} = y_{0}\,\,\,{\rm{and}}\,\,\,\lim_{j\rightarrow \infty}A_{i_{j}}^{-1} = x_{0}
$$
uniformly on compact subsets in $\bar{\mathbb{R}}^{n}\setminus\{x_{0}\}$ and $ \bar{\mathbb{R}}^{n}\setminus\{y_{0}\}$, respectively.

Since $x_{0}$ and $y_{0}$ are accumulation points of any orbit in $\mathbb{B}^{n}$, they are elements of the limit set of $\Delta$, which is contained in $\mathbb{S}^{n-1}$. Since $\Delta$ is a discrete elementary group, $x_{0}$ and $y_{0}$ are fixed points of elements in $\Delta$. 

Now $\{x_{0},y_{0}\} \cap \{N,S\} = \emptyset$ since $N$ and $S$ are not fixed points, while $x_{0}$ and $y_{0}$ are. Therefore, since $\lim_{j\rightarrow \infty}\alpha_{t(i_{j})}^{-1}(\vec{0}) = S$, $S \neq x_{0}$ and $\lim_{j \rightarrow \infty}A_{i_{j}} = y_{0}$ uniformly on compact subsets of $\overline{\mathbb{B}^{n}}\setminus\{x_{0}\}$, we have that $\lim_{j\rightarrow \infty}A_{i_{j}}\alpha_{t(i_{j})}^{-1}(\vec{0}) = y_{0}$. Hence, $\lim_{j\rightarrow\infty}\alpha_{t(i_{j})}A_{i_{j}}\alpha_{t(i_{j})}^{-1}(\vec{0}) = N$. Thus, transferring from the ball model to the hyperboloid model, we have \[\lim_{j\rightarrow \infty}|\alpha_{t(i_{j})}A_{i_{j}}\alpha_{t(i_{j})}^{-1}(e_{1})| = \infty.\] Therefore $$\lim_{j\rightarrow \infty}\|\alpha_{t(i_{j})}A_{i_{j}}\alpha_{t(i_{j})}^{-1}\| = \infty.$$\end{proof}


\emph{Proof Proposition~\ref{origthm2}:} To complete the proof  of Proposition~\ref{origthm2} we observe that by Lemma~\ref{keylemma}

\begin{eqnarray*} 
\left|\|\alpha_{t(i)}A_{i}\alpha_{t(i)}^{-1}\| - 1\right| & = & \left|\|\alpha_{t(i)}A_{i}\alpha_{t(i)}^{-1}\| - \|I_{n+1}\|\right|\\
 & \leq & \|\alpha_{t(i)}A_{i}\alpha_{t(i)}^{-1} - I_{n+1}\|\\ & < &c_{k} < 1
  \end{eqnarray*}
  
  which implies that 
  $$
  \|\alpha_{t(i)}A_{i}\alpha_{t(i)}^{-1}\| < 2
  $$
  
 which contradicts Claim~\ref{claim2}.$\square$
 
\section{Proof of Main Result}
\label{OrigThm3}

In this section we prove several technical lemmas. These lemmas will be used show that if the conclusion of Proposition~\ref{origthm2} holds we can establish an upper bound on the number of elements of $\Gamma$ that fail to move a ball of radius $r$ off itself. This result will allow us to bound the volume of the image of such a ball in the orbit space, thereby  establishing Theorem~\ref{MainTheorem}.


 
 First , we see that a bound on the translation distance of a hyperbolic isometry implies a bound on the entries of the associated matrix.
 
 \begin{lemma}
 \label{31}
 Given $r > 0$, if $A \in O^{+}(1,n)$ is such that $A\left(\overline{B(e_1,r)}\right) \bigcap \overline{B(e_1,r)} \neq \emptyset$, then 
  $$ 
 |a_{ij}| \leq \left(\frac{1 + \cosh r}{\sinh r}\right)\sqrt{\cosh 6r} = \kappa(r)\mbox{   for all $i,j$.}
$$\end{lemma}

\begin{proof} For all $j \neq1$ and for all $i$
\begin{eqnarray*}
|a_{ij}| &\leq& |Ae_{j}|\\ &=& |A(\frac{\cosh r}{\sinh r}e_{1}+e_{j}-\frac{\cosh r}{\sinh r}e_{1})|\\ &\leq& \frac{1}{\sinh r}|A((\cosh r)e_{1} + (\sinh r)e_{j})| + \frac{\cosh r}{\sinh r}|Ae_{1}|
\end{eqnarray*}
The assumption that $A(\overline{B(e_1,r)}) \bigcap \overline{B(e_1,r)} \neq \emptyset$, implies that $A(\overline{B(e_1,r)}) \subset \overline{B(e_1,3r)}$. Now, let $ v=(\cosh r)e_{1} + (\sinh r)e_{j}$. Then $\cosh d_{\mathbb{H}}(v,e_{1}) = -v\circ e_{1} = \cosh r$. Since $r$ and $d_{\mathbb{H}}(v,e_{1})$ are non-negative we have that $d_{\mathbb{H}}(v,e_{1})=r$. Therefore $v\in\overline{B(e_1,r)}$. Clearly $e_{1} \in \overline{B(e_1,r)}$. Furthermore, as any element of $\overline{B(e_1,3r)}$ has euclidean length at most $\sqrt{\sinh^{2} 3r + \cosh^{2} 3r}$, we have:

$$
|a_{ij}| \leq \left( \frac{1}{\sinh r} + \frac{\cosh r}{\sinh r}\right) \sqrt{\sinh^{2} 3r + \cosh^{2} 3r} = \left(\frac{1 + \cosh r}{\sinh r}\right)\sqrt{\cosh 6r}
$$

for all $i,j \neq 1$. For $j =1$:
\begin{eqnarray*}
|a_{i1}| &\leq& |Ae_{1}|\\ &\leq& \sqrt{\sinh^{2} 3r + \cosh^{2} 3r}\\ &=& \sqrt{\cosh 6r}
\end{eqnarray*}

Since $\frac{1+ \cosh r}{\sinh r} > 1$, the result follows.\end{proof}

A bound on the entries of a matrix $A$ implies a bound on its operator norm.

\begin{lemma}
\label{32} Let $A$ be any $(n+1)\times(n+1)$ matrix. Given $d>0$, if $|a_{ij}| \leq d$ for all i and j, then 

$$
\|A\| \leq d(n + 1)
$$\end{lemma}

\begin{proof} Let $v$ be a vector of unit length. Let $A_{i}$ denote the $i^{th}$ row  of the matrix $A$. Then

\begin{eqnarray*}
|A_{i} \cdot v| &\leq& |A_{i}||v|\\ &=& |A_{i}| \\ &=& \sqrt{\sum_{j=1}^{n+1} (a_{ij})^{2}} \\ &\leq& \sqrt{\sum_{j=1}^{n+1} d^{2}} \\&=& d\sqrt{(n+1)}
\end{eqnarray*}

Hence

\begin{eqnarray*}
|Av| &=& |(A_{1}\cdot v, \cdots ,A_{n+1}\cdot v)| \\ &=& \sqrt{(A_{1} \cdot v)^{2}+ \cdots+ (A_{n+1} \cdot v)^{2}}\\ &\leq& \sqrt{(n+1)(n+1)d^{2}}\\ &=& d(n+1)
\end{eqnarray*}

Since $v$ was chosen arbitrarily, by our definition of the operator norm $$\|A\| \leq d(n+1).$$\end{proof}

Given two matrices $A$ and $B$, a bound on the distance between corresponding entries gives a bound on $\|AB^{-1}-I_{n+1}\|$.

\begin{lemma}
\label{33} Given $K,L > 0$, let $\delta : = \frac{L}{(n+1)^{2}K}$. For $A,B \in O^{+}(1,n)$, if $|b_{ij}| \leq K$ and $|a_{ij} - b_{ij}| < \delta$ for all $i,j$, then $$ \|AB^{-1} - I_{n+1}\| \leq L$$\end{lemma}

\begin{proof}
\begin{eqnarray*}
\|AB^{-1} - I_{n+1}\| &=& \|(A - B)B^{-1}\|\\ &\leq& \|(A - B)\| \|B^{-1}\| \\ &=&\|(A - B)\| \|B\| \end{eqnarray*}

Therefore, by Lemma~\ref{32} $$\|AB^{-1} - I_{n+1}\| \leq (n+1)\delta(n+1)K \leq (n+1)^{2}\delta K$$\end{proof}

The following lemma bounds the size of a bounded uniformly discrete subset of $\mathbb{R}^{p}$.
\begin{lemma}
\label{34}
 Let $p \in \mathbb{Z}^{+}$ and $q,s \in \mathbb{R}^{+}$ be given. Let $M \subset \mathbb{R}^{p}$ be such that:
\begin{eqnarray*}
               &(i)& (a_i) \in M \mbox{ implies that } |a_i| \leq q \mbox{ for all } i\\
               \mbox{and} &(ii)& (a_i),(b_i) \in M \mbox{ implies that } |a_{i} - b_{i}| > s \mbox{ for some } i
               \end{eqnarray*}
Then
$$
|M| \leq \left(\frac{2q}{s} + 1\right)^{p}
$$\end{lemma}

\begin{proof} We divide the interval $[-q,q]$ as follows

\begin{eqnarray*}
E_{1} &=& [-q,-q+s)\\
E_{2} &=& [-q+s,-q+2s)\\
&\vdots&\\
E_{[\frac{2q}{s}]+1} &=& \left[-q+\left[\frac{2q}{s}\right]s,q\right]
\end{eqnarray*}

Now each $(a_i) \in M$ is an element of a $p$-cylinder $E_{a} = E_{j1} \times E_{j2} \times \cdots \times E_{jp}$. Where $(a_i) \in E_{ji}$, for all $i$. We note that by condition (ii) above, if $(a_i) \in E_{a} \mbox{ and } (a_i)\neq(b_i) \mbox{ then } (b_i) \not\in E_{a}.$

Hence we need only count the number of possible cylinders. As there are a maximum of $\frac{2q}{s}+1$ choices for $p$ different positions, our conclusion follows.\end{proof}

We now use our series of lemmas to bound the set of elements of a discrete group of hyperbolic isometries that fail to move a ball of radius $r$ of of itself. 

\begin{definition}
Define $\mathcal{H}(r,\Gamma) := \{ A \in \Gamma: A(\overline{B(e_1,r)}) \bigcap \overline{B(e_1,r)}\neq \emptyset \}.$
\end{definition}

\begin{lemma}
\label{35}
 Let $\Gamma$ be a subgroup of $O^{+}(1,n)$ such that $\|A - I_{n+1}\| \geq c_{k}$ for all $A \in \Gamma - I_{n+1}$. For $r >0$
$$
|\mathcal{H}(r,\Gamma)| \leq \left(\frac{2\kappa(r)^{2}(n+1)^{2}}{c_{k}} + 1\right)^{(n+1)^{2}}
$$ \end{lemma}

\begin{proof} Let $r > 0$ be given. Let $A$ and $B$ be distinct elements of $\mathcal{H}(r,\Gamma)$. Then $A$ and $B$ are distinct elements of $\Gamma$ and, therefore, $AB^{-1}$ is a non-identity element of $\Gamma$. So by assumption, we have:
$$
\|AB^{-1} - I_{n+1}\| \geq c_{k}
$$
 By Lemma~\ref{31}, $|b_{ij}| \leq \kappa(r)$ for all $i,j$. Thus by Lemma~\ref{33}, there exists $\hat{i},\hat{j}$, such that 

$$
|a_{\hat{i}\hat{j}} - b_{\hat{i}\hat{j}}| > \frac{c_{k}}{(n+1)^{2}\kappa(r)}
$$

We now apply Lemma~\ref{34}, with $p = (n+1)^{2}$, $q    = \kappa(r)$ and $s = \frac{c_k}{(n+1)^{2}\kappa(r)}$. In this case:
$$
\left(\frac{2q}{s} + 1\right)^{p} = \left(\frac{2\kappa(r)^{2}(n+1)^{2}}{c_{k}} + 1\right)^{(n+1)^{2}}
$$
\end{proof}


We are now prepared to prove our main result, which for convenience is restated below.

\begin{theorem}[Main Theorem]
Let $\Gamma$ be a discrete group of orientation-preserving isometries of $\mathbb{H}^n$. Assume that $\Gamma$ has no torsion element of order greater than $k$. Then
\[\Vol(\mathbb{H}^n/\Gamma) \geq \mathcal{A}(n,k)\]
where $\mathcal{A}(n,k)$ is the constant depending only on $n$ and $k$ defined in the introduction.
\end{theorem}

\begin{proof} Let $\Gamma$ be a discrete subgroup of $SO^{+}(1,n)$ which contains no torsion elements of order greater than $k$. If $\Gamma$ is an elementary group, its co-volume would be infinite.  Therefore, we may assume that $\Gamma$ is non-elementary. 
Hence by Proposition~\ref{origthm2}, there exists a group $\Gamma^{\prime}$, conjugate to $\Gamma$, such that $$\|A-I_{n+1}\| \geq c_{k} $$
for all $A \in \Gamma^{\prime}$,\,$A\neq I_{n+1}$.

Let $\pi$ be the covering projection from hyperbolic $n$-space onto $Q = \mathbb{H}^n/\Gamma^{\prime}$. For $r>0$, let $\mathcal{H} := \{ A \in \Gamma^{\prime}: A(\overline{B(e_1,r)}) \bigcap \overline{B(e_1,r)}\neq \emptyset \}.$

The map $\pi$ restricted to $B(e_1,r)$ is a local isometry away from the singular locus of $Q$. Notice that the singular locus has volume zero. If $x \in \pi(B(e_1,r))$ then $$|\pi^{-1}(x) \cap B(e_1,r)| \leq |\mathcal{H}|.$$
Therefore
\begin{eqnarray*}
\Vol\left(\frac{\mathbb{H}^{n}}{\Gamma}\right) &=& \Vol\left(\frac{\mathbb{H}^{n}}{\Gamma^{\prime}}\right) \\ &\geq& \Vol(\pi(B(e_1,r))\\&\geq& \frac{\Vol(B(e_1,r))}{|\mathcal{H}|}
\end{eqnarray*}

And our result follows from Lemma~\ref{35}. \end{proof}

\section{Corollaries}
Corollary~\ref{corollary 1} follows readily from Theorem~\ref{MainTheorem}. We recall that the quotient $Q$ of a hyperbolic manifold $M$ by its group of orientation-preserving isometries $G$ is an orientable hyperbolic orbifold (as long as $\pi_1(M)$ is not virtually abelian, in which case $\Vol(M)$ is infinite and Corollary~\ref{corollary 1} is vacuous). Note that \[\Vol(Q)=\frac{\Vol(M)}{|G|}.\] Since, under the assumptions of Corollary~\ref{corollary 1}, $\Vol(Q)\ge\mathcal{A}(n,k)$, we obtain our result.
\medskip

The Mostow-Prasad rigidity theorem (\cite{mostow},\cite{prasad}) implies that if $M$ has finite volume then one can identify the group of isometries of $M$ with $\Out(\pi_1(M))$. With this in mind, we can give the following more topological version of Corollary~\ref{corollary 1}.

\begin{corollary}
\label{corollary 2}
If $M$ is a finite volume orientable hyperbolic n-manifold and $G$ is a subgroup of $\Out(\pi_1(M))$ containing no torsion elements of order greater than $k$, then \[|G| \leq\frac{2\Vol(M)}{\mathcal{A}(n,k)}.\]
\end{corollary}
\bigskip

\nocite{*}
\bibliography{VolumesOne}
\end{document}